\documentclass[11pt]{amsart}
\usepackage[utf8]{inputenc}
\usepackage{amssymb}
\usepackage{amsfonts}
\usepackage{amsmath}
\usepackage{xcolor}
\usepackage[english]{babel}
\usepackage{amsthm}
\usepackage{setspace}
\usepackage{tikz-cd}
\usepackage{graphicx}
\usepackage[perpage]{footmisc}
\usepackage[shortlabels]{enumitem}
\usepackage{longtable}
\usepackage{comment}
\usepackage{diagbox}
\usepackage{bbm}
\usepackage{stmaryrd}

\usepackage{mathtools}

\newtheoremstyle{example}
{}                
{}                
{\sffamily}        
{}                
{\bfseries}       
{.}               
{ }               
{}                

\usepackage[
backend=biber,
style=alphabetic
]{biblatex}
\usepackage{biblatex}
\usepackage{hyperref}
\usepackage{xurl}
\hypersetup{breaklinks=true}

\newtheoremstyle{example}
{}                
{}                
{\sffamily}        
{}                
{\bfseries}       
{.}               
{ }               
{}                

\usepackage[
backend=biber,
style=alphabetic
]{biblatex}
\newtheorem{theorem}{Theorem}

\newtheorem{lemma}[theorem]{Lemma}
\newtheorem{corollary}[theorem]{Corollary}

\theoremstyle{definition}
\newtheorem{definition}[theorem]{Definition}

\newtheorem{remark}[theorem]{Remark}
\newtheorem{example}[theorem]{Example}

\newtheorem{conjecture}[theorem]{Conjecture}

\numberwithin{theorem}{section}

\numberwithin{theorem}{section}

\DeclareMathOperator{\Hom}{Hom}
\DeclareMathOperator{\id}{id}
\DeclareMathOperator{\im}{im}
\DeclareMathOperator{\GL}{GL}

\DeclareMathOperator{\End}{End}

\DeclareMathOperator{\Gal}{Gal}

\DeclareMathOperator{\Nm}{Nm}

\DeclareMathOperator{\Br}{Br}
\DeclareMathOperator{\NS}{NS}
\DeclareMathOperator{\Pic}{Pic}

\DeclareMathOperator{\rad}{rad}
\DeclareMathOperator{\MT}{MT}

\newcommand{\co}{\mathcal{O}}

\newcommand{\Z}{\mathbb{Z}}
\newcommand{\R}{\mathbb{R}}
\newcommand{\Q}{\mathbb{Q}}
\newcommand{\C}{\mathbb{C}}
\newcommand{\F}{\mathbb{F}}

\newcommand{\p}{\mathfrak{p}}

\newcommand{\A}{\mathbb{A}}

\newcommand{\Hdg}{\mathrm{Hdg}}

\usepackage[margin=1.2in]{geometry}
\usepackage[style=alphabetic]{biblatex}

\title[Brauer groups of principal CM K3 surfaces]{Explicit bounds on the transcendental Brauer group of K3 surfaces with principal complex multiplication}
\author{Sebastian Monnet}

\begin{document}
\begin{abstract}
    Let $X$ be a K3 surface defined over a number field $k$, with principal complex multiplication by a CM field $E$. We find explicit bounds, in terms of $k$ and $E$, on the size of the transcendental Brauer group $\Br(X)/\Br_1(X)$ of $X$. Bounding the size of this group is important for computing the Brauer--Manin obstruction, which is conjectured by Skorobogatov to be the only obstruction to the Hasse principle for K3 surfaces. Our methods are built on top of earlier work by Valloni, who related the group $\Br(X)/\Br_1(X)$ to the arithmetic structure of the CM field $E$. It is from this arithmetic structure that we deduce our bounds.    
\end{abstract}

\maketitle
\tableofcontents

\clearpage
\section{Introduction}
\label{sec-intro}
In the rational points community, K3 surfaces are one of the most studied types of variety. They are simultaneously nice enough to work with and nasty enough that there are plenty of open problems. One such open problem is the following conjecture by Skorobogatov, which can be found in \cite[Page~1859]{skorobogatov-conjecture}:
\begin{conjecture}[Skorobogatov]
    \label{conj-skorobogatov-k3}
The Brauer--Manin obstruction is the only obstruction to the Hasse principle for K3 surfaces.
\end{conjecture}
If Conjecture~\ref{conj-skorobogatov-k3} is true, then for a K3 surface $X$ over a number field $k$, we can determine whether $X$ has $k$-rational points by computing its Brauer--Manin set $X(\A_k)^{\Br(X)}$. In light of \cite[Theorem~1]{kresch-tschinkel}, this computation can be performed effectively\footnote{Meaning that its running time can be bounded explicitly in terms of $X$.} if the size of $\Br(X) / \Br_0(X)$ can be bounded effectively, where $\Br(X)$ is the Brauer group of $X$ and $\Br_0(X)$ is the image of the natural map $\Br(k) \to \Br(X)$. Fix an algebraic closure $\overline{k}$ of $k$, let $G_k = \Gal(\overline{k}/k)$, and let $\overline{X}$ be the base change $X_{\overline{k}}$ of $X$ to $\overline{k}$. Writing
$$
\Br_1(X) = \ker\Big(\Br(X) \to \Br(\overline{X})^{G_k}\Big),
$$
we have subgroups 
$$
\Br_0(X) \subseteq \Br_1(X) \subseteq \Br(X). 
$$
The algebraic Brauer group $\Br_1(X)/\Br_0(X)$ has already been uniformly bounded (see \cite[Remark~1.4]{bound-on-algebraic-brauer}), so it suffices to find an effective bound on the transcendental Brauer group $\Br(X) / \Br_1(X)$. In the case where $X$ has complex multiplication by the ring of integers of a CM field $E$, Valloni \cite{valloni-algorithm} gives an algorithm for computing a bound on the size of $\Br(X) / \Br_1(X)$ in terms of $E$ and the field of definition of $X$. This algorithm has not been implemented for general $E$, so in practice it is useful to have an explicit bound on the sizes it may produce. 

In this paper, we analyse Valloni's algorithm to deduce an explicit bound on its output. Let $E$ be a CM field and let $F$ be its maximal totally real subfield. Using the standard notation set out in Section~\ref{subsec-nf-notation}, we define
$$
M_E = \frac{[\co_E^\times : \co_F^\times]h_F}{h_E} \cdot 2^{\omega(d_{E/F}) + \omega(d_{E/F,2})} \cdot \sqrt{\Nm(d_{E/F,2})}.
$$
We obtain the following explicit bound on the transcendental Brauer group:
\begin{theorem}
    \label{thm-explicit-bound-without-phi-and-psi}
    Let $k$ be a number field and let $E$ be a CM field with maximal totally real subfield $F$. Then every K3 surface $X/k$ with CM by $\co_E$ has
    $$
    \#\big(\Br(X)/\Br_1(X)\big) \leq 3^{[E:\Q]}\cdot \Nm(\rad(d_{E/F})) \cdot ([k:\Q]M_E)^{2\log_2 3}.
    $$ 
\end{theorem}
In fact, we can do much better than Theorem~\ref{thm-explicit-bound-without-phi-and-psi}, at the expense of a more complicated statement. In order to state our sharper bound, we define functions $\Phi : \Z_{\geq 0} \to \R_{\geq 0}$ and $\Psi : \Z_{\geq -1} \to \R_{\geq 0}$ as follows. For each nonnegative integer $i$, write $p_i$ for the $i^\mathrm{th}$ prime, so that e.g. $p_0 = 2$ and $p_1 = 3$. For each integer $n\geq 0$, define  
$$
\Phi(n) = \sum_{i = 0}^{n} \log(p_i - 1) 
$$
and 
$$
\Psi(n) = \prod_{i=0}^{n+1} \frac{p_i}{p_i - 1}. 
$$
Since $\Phi(n)$ is monotonically increasing and $\Phi(0) = 0$, we may define a ``pseudoinverse'' function $\Phi^{-1} : \R_{\geq 0} \to \Z_{\geq 0}$ by 
$$
\Phi^{-1}(t) = \max\{n \in \Z_{\geq 0} : \Phi(n) \leq t\}.
$$
The notation $\Phi^{-1}$ comes from the fact that $\Phi^{-1} \circ \Phi = \id_{\Z_{\geq 0}}$, whilst composing the other way around does not give the identity. Our sharpest bound on $\Br(X)/\Br_1(X)$ is as follows:

\begin{theorem}
    \label{thm-bound-with-Phi-and-Psi}
    Let $k$ be a number field and let $E$ be a CM field. Then every K3 surface $X/k$ with CM by $\co_E$ has 
    $$
    \#\big(\Br(X)/\Br_1(X)\big) \leq [k : \Q]^2 \cdot M_E^2 \cdot \Psi\Big(\Phi^{-1}\Big(
        \frac{2\log([k:\Q]M_E)}{[E:\Q]}
    \Big)\Big)^{[E:\Q]} \cdot \Nm(\rad(d_{E/F})).
    $$
\end{theorem}

The functions $\Phi$ and $\Psi$ are defined explicitly so, given the fields $k$ and $E$, we may compute the quantity 
$$
\Psi\Big(\Phi^{-1}\Big(
        \frac{2\log([k:\Q]M_E)}{[E:\Q]}
    \Big)\Big),
$$
making the bound of Theorem~\ref{thm-bound-with-Phi-and-Psi} into an explicit number, as in the following example:
\begin{example} 
    Taking $k = \Q$ and $E = \Q(i)$, we obtain 
    $
    M_E = 16,
    $
    and hence 
    $$
    \Psi\Big(\Phi^{-1}\Big(
        \frac{2\log([k:\Q]M_E)}{[E:\Q]}
    \Big)\Big) = \Psi\Big(\Phi^{-1}\Big(\log(16)\Big)\Big) = \frac{35}{8}.
    $$ 
    Theorem~\ref{thm-bound-with-Phi-and-Psi} then gives 
    $$
    \#\big(\Br(X)/\Br_1(X)\big) \leq 9800.
    $$
    Using his algorithm in \cite[Section~11.1]{valloni-algorithm}, Valloni obtains a bound of $100$ for the same quantity. Valloni's bound is better because working explicitly in $\Q(i)$ allows greater precision. Our proofs account for possible cases that Valloni rules out explicitly. Thus, there is still significant value in performing Valloni's computations in full detail. However, our bound is explicit and much easier to compute, so it adds value, particularly when computational efficiency is important. 
\end{example}

The bound in Theorem~\ref{thm-bound-with-Phi-and-Psi} looks a bit opaque because of the functions $\Phi^{-1}$ and $\Psi$, whose behaviours are not immediately clear. Theorem~\ref{thm-explicit-bound-without-phi-and-psi} is obtained using a very crude bound on the asymptotics of $\Psi \circ \Phi^{-1}$. The crude bounding method in question can be refined arbitrarily far, to obtain the following result:
\begin{theorem}
    \label{thm-bounding-Psi-Phi-inverse}
    Let $\delta >0$ be any positive real number, and let $L_\delta$ be an integer with $L_\delta \geq - 1$ and $\frac{p_{L_\delta + 2}}{p_{L_\delta + 2} - 1} \leq 1 + \delta$. Then we have 
    
    $$
    \Psi(\Phi^{-1}(t)) \leq \Psi(L_\delta)\cdot  (1 + \delta)^{1 + \frac{t}{\log 2}},
    $$
    for all $t \geq 0$. 
\end{theorem}
We deduce the following asymptotic bound on the transcendental Brauer group:
\begin{corollary}
    \label{cor-to-bound-with-Psi-Phi-inverse}
    Let $\delta >0$ and let $L_{\delta}$ be as in Theorem~\ref{thm-bounding-Psi-Phi-inverse}. Then we have
    $$
    \#\big(\Br(X)/\Br_1(X)\big) \leq \big((1+\delta)\Psi(L_\delta)\big)^{[E: \Q]}\cdot \Nm(\rad(d_{E/F})) \cdot ([k : \Q]M_E)^{2(1 + \log_2(1+\delta))}.
    $$
\end{corollary}
Given $\delta$, it is natural to ask how efficiently we may compute the smallest possible $L_\delta$, and hence the best possible constant $\Psi(L_\delta)$. The following result gives us the answer:
\begin{theorem}
    \label{thm-time-complexity-C-eps-delta}
    Let $\delta > 0$, and let $L_\delta$ be the smallest integer with $L_\delta \geq -1$ and $\frac{p_{L_\delta + 2}}{p_{L_\delta + 2} - 1} \leq 1 + \delta$. The value of $L_\delta$, and the corresponding constant $\Psi(L_\delta)$, can be computed with time complexity 
    $$
    o\Big(\frac{1}{\delta}\Big),
    $$
    as $\delta \to 0$. 
\end{theorem}

\begin{example}
    [Imaginary quadratic fields]
    We will explore the implications of our bounds for K3 surfaces with CM by rings of integers of imaginary quadratic fields. Let $E = \Q(\sqrt{d})$ for a squarefree negative integer $d$. The cases $d = -1, -2, -3$ are easy but slightly irregular, so for simplicity we will assume that $d \leq -5$. Then we have 
    $$
    M_E = \begin{cases}
        \frac{1}{h_E}\cdot 2^{\omega(d)}\quad\quad\hspace{0.16em}\text{if $d\equiv 1\pmod{4}$},
        \\
        \frac{1}{h_E}\cdot 2^{\omega(d) + \frac{5}{2}}\quad\text{if $d \equiv 2\pmod{4}$},
        \\
        \frac{1}{h_E}\cdot 2^{\omega(d) + 3}\quad\hspace{0.16em}\text{if $d \equiv 3\pmod{4}$},
    \end{cases}
    $$
    so we always have
    $$
    M_E \leq \frac{1}{h_E} \cdot 2^{\omega(d)+ 3}.
    $$
    Let $\eta = \log_2(1+\delta)$. Then Corollary~\ref{cor-to-bound-with-Psi-Phi-inverse} tells us that 
    $$
    \#\big(\Br(X)/\Br_1(X)\big) \leq \Big(2\cdot(1+\delta)^2\cdot \Psi(L_\delta)^2\cdot 64^{1 + \eta}\Big) \cdot [k : \Q]^{2(1+\eta)}\cdot 4^{\omega(d)(1+\eta)} \cdot \lvert d \rvert. 
    $$
    It is easy to see\footnote{e.g. by showing that $\frac{\alpha^{\omega(d)}}{d}\to 0$ as $d\to\infty$ for all $\alpha > 0$, and then taking $\alpha = 4^{\frac{1+\eta}{\eta}}$.} that $4^{\omega(d)} = o\Big(\lvert d\rvert^{\frac{\eta}{1+\eta}}\Big)$ as $\lvert d\rvert\to\infty$, so we have 
    \begin{equation}
        \label{eqn-little-oh-bound}
    \#\big(\Br(X)/\Br_1(X)\big) = o\Big([k:\Q]^{2(1+\eta)}\cdot \lvert d\rvert^{1+\eta}\Big).
    \end{equation}
    Clearly $\eta \to 0$ as $\delta \to 0$, so Equation~(\ref{eqn-little-oh-bound}) holds for any $\eta > 0$. 
    So in other words, when $X$ has principal CM by an imaginary quadratic field, the bound on $\#(\Br(X)/\Br_1(X))$ grows very slightly faster than quadratically in the degree of the field of definition, and very slightly faster than linearly in the discriminant of the CM field. 
\end{example}

\section{Background}
In this section, we sketch a brief overview of the theory required to understand our results. For a proper introduction to K3 surfaces, we point the reader to \cite{Huybrechts_2016}. The theory of complex multiplication of K3 surfaces is studied in \cite{Zarhin1983}, and the details specific to our situation are developed in \cite{valloni-algorithm}. We also recommend Davide Lombardo's PhD thesis \cite{lombdardo-thesis}, which contains clear and detailed exposition of Hodge structures and Mumford--Tate groups.
\subsection{Notation for number fields}
\label{subsec-nf-notation}

We collect some standard notation, which is used throughout the paper. Given a number field $M$, write $\co_M$ for its ring of integers and $h_M$ for its class number. For an ideal $I \subseteq \co_M$, write $\omega(I)$ for the number of distinct prime ideals of $\co_M$ dividing $I$, and $\Nm(I)$ for its norm $\# (\co_M/I)$. Moreover, write $I_2$ for the $2$-part of $I$, which is defined to be 
$$
I_2 = \prod_{\p \mid 2} \p^{v_\p(I)},
$$
where the product is over prime ideals of $\co_M$ containing $2$. Write $\phi_M$ for the Euler totient function on ideals of $\co_M$, which is defined by 
$$
\phi_M(I) = \#(\co_M/I)^\times  
$$
for each ideal $I$ of $\co_M$. 
Given a CM field (i.e. a totally imaginary quadratic extension of a totally real number field) $E$, we will always write $F$ for its maximal totally real subfield, leaving $E$ implicit in the notation. Moreover, we write $d_{E/F}$ for the discriminant of the extension, which is an ideal of $\co_F$. Given a CM field $E$ and an element $x \in E$, we write $\overline{x}$ for the conjugate of $x$ over $F$. Similarly, for an ideal $I$ of $\co_E$, we write $\overline{I}$ for the conjugate ideal $\{\overline{x} : x \in I\}$. Given an implicit choice of CM field $E$, write $G$ for the Galois group $\Gal(E/F)$. 

Again when $E$ is a CM field, for each prime ideal $\p$ of $\co_F$, write $e(\p)$ for the ramification index of $\p$ in $E$. For an archimedean place $v$ of $F$, define $e(v)$ to be $2$ if $v$ can be extended to a place $w$ of $E$ with $[E_w : F_v] = 2$ (so that the extension is $\C/\R$), and set $e(v) = 1$ otherwise. When $\p$ is a prime ideal of $\co_F$ lying over $2$, write $e(\p\mid 2)$ for the ramification index of $\p$ over the rational prime $2$.

\subsection{Hodge structures}
\begin{definition}
    Let $R$ be one of $\Z$ and $\Q$, and let $V$ be a free $R$-module of finite rank. Write $V_\C$ for the base change $V\otimes_R \C$. A \emph{Hodge structure on $V$ of weight $n$} is a direct sum decomposition 
    $$
    V_\C = V^{n,0}\oplus V^{n-1, 1}\oplus \ldots \oplus V^{0,n}
    $$
    of vector spaces, such that the complex conjugation map $\iota:V_\C \to V_\C$ has $\iota(V^{p,q}) = V^{q,p}$ for each $(p,q)$. Let $V$ and $W$ be free $R$-modules with Hodge structures of weight $n$. A \emph{morphism of Hodge structures} $V \to W$ is an $R$-module homomorphism $\varphi : V \to W$ such that the natural base change $\varphi_\C : V_\C \to W_\C$ has 
    $$
    \varphi_\C(V^{p,q}) \subseteq W^{p,q}
    $$
    for each $(p,q)$. 
\end{definition}
The best-known example of Hodge structures comes from complex algebraic geometry. The following result can be found in \cite[Page~38]{Huybrechts_2016}:
\begin{theorem}
    \label{thm-hodge-structure-on-complex-variety}
    Let $X$ be a smooth projective variety defined over $\C$, and write $\Omega_X$ for its cotangent sheaf. For each positive integer $r$, there is a canonical Hodge structure of weight $r$ on $H^r(X,\Z)$ (respectively $H^r(X,\Q)$), given by 
    $$
    H^r(X,\C) = \bigoplus_{p+q=r} H^{p,q},
    $$
    where $H^{p,q}$ is the space 
    $$
    H^{p,q} = H^q(X, \Omega_X^p).
    $$
\end{theorem}

There is a particular real algebraic group, called the \emph{Deligne torus} and denoted by $\mathbb{S}$, such that Hodge structures are equivalent to representations of $\mathbb{S}$. In other words, for any rational vector space $V$, there is a natural bijection 
$$
\{\text{Hodge structures on $V$}\} \longleftrightarrow \Hom(\mathbb{S}, \GL(V_\R)),
$$
where the right-hand side denotes morphisms of real algebraic groups. Given a rational Hodge structure $V$, we thus obtain a homomorphism 
$$
\rho : \mathbb{S} \to \GL(V_\R) 
$$
of real algebraic groups. Taking real points, we get a group homomorphism 
$$
\rho_\R : \mathbb{S}(\R) \to \GL(V)(\R). 
$$ 
We define the \emph{Mumford--Tate group} $\MT(V)$ of $V$ to be the the smallest algebraic subgroup of $\GL(V)$ (defined over $\Q$) such that 
$$
\im(\rho_\R) \subseteq \MT(V)(\R). 
$$
\subsection{K3 surfaces and complex multiplication}

Let $k$ be a field and let $X$ be a variety over $k$. Write $\omega_X$ for the canonical sheaf of $X$ and $\co_X$ for the sheaf of regular functions on $X$.
\begin{definition}
    Let $k$ be a field. A \emph{K3 surface over $k$} is a complete nonsingular surface $X$ over $k$, such that $\omega_X \cong \co_X$ and $H^1(X,\co_X) = 0$. 
\end{definition}

Let $X$ be a K3 surface over $\C$. Then there is a short exact sequence of sheaves, called the \emph{exponential exact sequence}, given by 
$$
0 \to \underline{\Z} \to \co_X \overset{\exp}{\to} \co_X^{\times} \to 1. 
$$
The associated long exact sequence of cohomology gives an exact sequence 
\begin{equation}
    \label{eqn-exact-sequence-k3-surface}
0 \to \Pic(X) \to H^2(X,\Z) \to \C \to H^2(X,\co_X^\times) \to H^3(X,\Z) \to 0. 
\end{equation}
Since $\Pic(X)$ and $\C$ are torsion-free, it follows that $H^2(X,\Z)$ is also torsion-free. From basic properties of integral cohomology, one can obtain $H^3(X,\Z) = H^1(X,\Z) = 0$. 

There is a natural \emph{intersection pairing} 
$$
H^2(X,\Z) \times H^2(X,\Z) \to H^4(X,\Z) \cong \Z,
$$
which makes $H^2(X,\Z)$ into a \emph{lattice}, by which we mean a free abelian group equipped with a nondegenerate bilinear form. For the formal definition and properties of lattices, see \cite[Chapter~14]{Huybrechts_2016}. The following result, which is \cite[Proposition~3.5]{Huybrechts_2016}, says that all K3 surfaces give rise to the same lattice:

\begin{theorem}
    There is a distinguished lattice $\Lambda_{\mathrm{K3}}$, called the \emph{K3 lattice}, such that for any K3 surface $X$ over $\C$, there is a lattice isomorphism
    $$
    H^2(X,\Z) \cong \Lambda_{\mathrm{K3}}.
    $$
\end{theorem}

Theorem~\ref{thm-hodge-structure-on-complex-variety} gives a natural Hodge structure 
$$
H^2(X,\C) = H^{0,2} \oplus H^{1,1} \oplus H^{2,0}
$$
on $H^2(X,\Z)$ and $H^2(X,\Q)$.
For any K3 surface $X$, we always have (see \cite[Page~8]{Huybrechts_2016})
$$
\dim H^{p,q} = \begin{cases}
    1\quad\text{if $(p,q) = (2,0)$ or $(0,2)$},
    \\
    20 \quad\text{if $(p,q) = (1,1)$},
    \\
    0\quad\text{otherwise}. 
\end{cases}
$$

Recall from Equation~(\ref{eqn-exact-sequence-k3-surface}) that there is a natural embedding 
$$
\Pic(X) \hookrightarrow H^2(X,\Z). 
$$
The image of this embedding is a sublattice of $H^2(X,\Z)$, which we call the \emph{N\'eron-Severi lattice} of $X$, denoted $\NS(X)$. We define the \emph{transcendental lattice} $T(X)$ of $X$ to be the orthogonal complement
$$
T(X) = \NS(X)^\perp,
$$
with respect to the intersection pairing on $H^2(X,\Z)$. The sublattice $T(X)\subseteq H^2(X,\Z)$ is well-behaved with respect to the Hodge structure on $H^2(X,\Z)$, so it inherits its own Hodge structure. The rational Hodge structure $T(X)_\Q$ has a well-defined ring $\End_{\Hdg}(T(X)_\Q)$ of Hodge endomorphisms. Denote this ring by $E(X)$. 
\begin{theorem}[Zarhin]
    \label{thm-abelian-MT-implies-CM-field}
    If the Mumford--Tate group $\MT(T(X)_\Q)$ is abelian, then $E(X)$ is isomorphic to a CM field, and the $E(X)$-vector space $T(X)_\Q$ has 
    $$
    \dim_{E(X)}T(X)_\Q = 1. 
    $$
\end{theorem}
\begin{proof}
    This is stated in \cite[Page~12]{valloni-algorithm}.
\end{proof}
\begin{definition}
In light of Theorem~\ref{thm-abelian-MT-implies-CM-field}, we say that a K3 surface $X$ over a number field has \emph{complex multiplication} if the Mumford--Tate group $\MT(T(X)_\Q)$ is abelian. In that case, we say that $X$ has complex multiplication \emph{by $E$}, where $E$ is any field isomorphic to $E(X)$. 
\end{definition}
Let $X$ be a K3 surface with complex multiplication. Then the ring $\End_{\Hdg}(T(X))$ is an order in $E(X)$, and we denote it by $\co(X)$. 
\begin{definition}
    Let $E$ be a CM field. A K3 surface $X$ has \emph{principal complex multiplication by $E$} if it is has CM by $E$ and $\co(X)$ is a maximal order in $E(X)$. Alternatively, we say that $X$ has complex multiplication \emph{by $\co_E$}.
\end{definition}

\subsection{K3 class groups}

Let $K$ be a number field and write $\mathbb{A}_{K,f}^\times$ for its finite ideles. Given an ideal $I$ of $K$, define
$$
U_{K,I} = \{s \in \A_{K,f}^\times : v(s) = 0 \text{ and } v(s-1) \geq v(I) \text{ for all finite places $v$}\}.
$$
The ray class group of $K$ modulo $I$ is isomorphic to the double quotient
$$
K^\times \backslash \A_{K,f}^\times / U_{K,I}. 
$$
The \emph{K3 class group} is a sort of ``CM-analogue'' of this situation. We basically replace the number field $K$ with a CM field $E$, and replace ``units of $K$'' with ``units of $E$ with norm $1$ over the totally real subfield''. Given a CM field $E$, define
$$
U(E) = \{e \in E^\times : e\bar{e} = 1\},\quad U(\A_{E,f}) = \{s \in \A_{E,f}^\times : s\bar{s} = 1\},
$$
and, for any ideal $I$ of $E$, 
$$
\widetilde{U}_{E,I} = \{s \in U(\A_{E,f}) : v(s) = 0 \text{ and } v(s-1) \geq v(I) \text{ for all finite places $v$ of $E$}\}. 
$$
Then the \emph{K3 class group of $E$ modulo $I$} is defined in \cite[Definition~9.5]{valloni-algorithm} to be the double quotient
$$
G_{\mathrm{K3}, I}(E) = U(E)\backslash U(\A_{E,f}) / \widetilde{U}_{E,I}.
$$
Valloni then defines the \emph{K3 class field of $E$ modulo $I$} to be the abelian extension $F_{K3,I}(E) / E$ associated to the surjection 
$$
\A_{E,f}^\times \to G_{\mathrm{K3}, I}(E),\quad s \mapsto \frac{s}{\overline{s}}. 
$$
\begin{theorem}
    [Valloni]
    \label{thm-existence-of-suitable-ideal}
    Let $K$ be a number field, let $E$ be a CM field with $E\subseteq K$, and let $X/K$ be a K3 surface with CM by $\co_E$. There is an ideal $I$ of $\co_E$ such that the following two statements are true:
    \begin{enumerate}
        \item $\Br(\overline{X})^{G_K} \cong \co_E/I$.
        \item $F_{\mathrm{K3},I}(E) \subseteq K$. 
    \end{enumerate}
\end{theorem}
\begin{proof}
    This is stated\footnote{Read the paragraph starting with ``To see how''.} in \cite[Pages~6-7]{valloni-algorithm}. The result is essentially \cite[Corollary~11.3]{valloni-algorithm}, but we have opted for the former reference since the equivalence 
    $$
    \Br(\overline{X})^{G_K} = \Br(\overline{X})[I] \cong \co_E/I
    $$ 
    is not immediate. 
\end{proof}
\begin{remark}
    \label{remark-level-structures-and-ideals}
    In the proof of Theorem~\ref{thm-existence-of-suitable-ideal}, we noted that it was not obvious to us that 
    $$
    \Br(\overline{X})^{G_K} = \Br(\overline{X})[I] \cong \co_E/I
    $$
    for an ideal $I$ of $\co_E$. Since this took us some time to understand, we sketch a brief explanation here. Let $X$ be a K3 surface defined over $\C$, with CM by $\co_E$. By \cite[Remarks~7.1(3)]{valloni-algorithm}, there is an isomorphism of $\co_E$-modules 
    $$
    E/I^\vee \to \Br(X) 
    $$
    for some nonzero fractional ideal $I^\vee$ of $E$. Thus, finite $\co_E$-invariant subgroups of $\Br(X)$ correspond bijectively to finite $\co_E$-invariant subgroups of $E/I^\vee$, which are precisely the quotients $J/I^\vee$ for fractional ideals $J$ of $E$ with 
    $
    I^\vee \subseteq J. 
    $
    For each ideal $\mathfrak{a}$ of $\co_E$, define 
    $$
    \Br(X)[\mathfrak{a}] = \{b \in \Br(X) : ab = 0 \text{ for all $a \in \mathfrak{a}$}\}.
    $$
    Since $\Br(X) \cong E/I^\vee$, it follows from the definitions that 
    $$
    \Br(X)[\mathfrak{a}] \cong \frac{\mathfrak{a}^{-1}I^\vee}{I^\vee}. 
    $$

    We have mutually inverse bijections 
    $$
    \begin{tikzcd}
    \{\text{ideals of $\co_E$}\} \arrow[rr, "\mathfrak{a}\mapsto \mathfrak{a}^{-1}I^\vee", shift left] & & \{\text{fractional ideals $J$ with $I^\vee \subseteq J$}\} \arrow[ll, "I^\vee J^{-1} \mapsfrom J", shift left],
        \end{tikzcd}
    $$
    so every level structure $B$ (by which Valloni means a finite, $\co_E$-invariant subgroup of $\Br(X)$) corresponds to a unique ideal $\mathfrak{a}$, with the properties that 
    $$
    B \cong \Br(X)[\mathfrak{a}] \cong \co_E/\mathfrak{a}. 
    $$
\end{remark}
Given a CM field $E$, call an ideal $I$ of $\co_E$ \emph{conjugation-invariant} if $\overline{I} = I$. Let $I$ be a conjugation-invariant ideal and write $J = I \cap \co_F$. We fix the following notation:
\begin{enumerate}
    \item $E^{I,1} = \{e \in E^\times : v(e-1) \geq v(I) \text{ for each finite place $v$ with $v(I) > 0$}\}$.
    \item $\co_E^I = \co_E^\times \cap E^{I,1}$.
    \item $H^1(E^{I,1}) = \hat{H}^1(G, E^{I,1})$, where $G = \Gal(E/F)$ and $\hat{H}$ denotes Tate cohomology. 
    \item $e(E/F,J) = \prod_{v\nmid J} e(v)$, where $v$ ranges over all places, finite and infinite. See Section~\ref{subsec-nf-notation} for the definition of $e(v)$. 
\end{enumerate} 
Then Valloni gives us the following formula:
\begin{theorem}[Valloni]
    \label{thm-formula-for-size-of-G-K3-I}
    Let $I$ be a conjugation-invariant, integral ideal of a CM field $E$. We have 
    $$
    \# G_{K3,I}(E) = \frac{2\cdot h_E \cdot \phi_E(I) \cdot [\co_F^\times : N_{E/F}(\co_E^I)]}{h_F \cdot \phi_F(J)\cdot [\co_E^\times : \co_E^I] \cdot e(E/F,J) \cdot \lvert H^1(E^{I,1})\rvert}. 
    $$
\end{theorem}
\begin{proof}
    This is \cite[Corollary~10.4]{valloni-algorithm}.
\end{proof}

\begin{definition}
    \label{defi-permissible-ideal}
    Let $k$ be a number field and let $E$ be a CM field. A \emph{$k$-permissible} ideal of $\co_E$ is a conjugation-invariant ideal $I$ of $\co_E$, such that 
    $$
    \frac{\phi_E(I)}{\phi_F(J)} \leq \frac{1}{2}\cdot [k : \Q] \cdot \frac{[\co_E^\times : \co_E^I]}{[\co_F^\times : N_{E/F}(\co_E^I)]} \cdot \frac{h_F}{h_E} \cdot e(E/F,J) \cdot \lvert H^1(E^{I,1})\rvert,
    $$
    where $J = I\cap \co_F$. 
\end{definition}

\begin{corollary}[Essentially due to Valloni]
    \label{cor-brauer-bounded-by-Nm-I}
    Let $k$ be a number field and let $E$ be a CM field. For any K3 surface defined over $k$ with CM by $\co_E$, there is a $k$-permissible ideal $I$ of $\co_E$ such that 
    $$
    \# \big(\Br(X) / \Br_1(X)\big) \leq \Nm(I).
    $$ 
\end{corollary}
\begin{proof}
    Let $X$ be a K3 surface over $k$ with CM by $\co_E$. Let $K = kE$. Then Theorem~\ref{thm-existence-of-suitable-ideal} tells us that there is an ideal $I_0$ of $\co_E$ with $F_{K3,I_0}(E)\subseteq K$ and $\Br(\overline{X})^{G_K} \cong \co_E/I_0$. Let $I = I_0 \cap \overline{I}_0$, so that $I$ is conjugation-invariant. On \cite[Page~40]{valloni-algorithm}, Valloni states that $G_{K3, I}(E) = G_{K3,I_0}(E)$. Clearly $\Nm(I) \geq \Nm(I_0)$, so we have 
    $$
    \# \Br(\overline{X})^{G_K} \leq \Nm(I) 
    $$
    and 
    $$
    E \subseteq F_{K3, I}(E) \subseteq K. 
    $$
    It follows that 
    $$
    \# G_{K3, I}(E) \mid [K : E] \leq [k : \Q],
    $$
    so Theorem~\ref{thm-formula-for-size-of-G-K3-I} implies that $I$ is $k$-permissible.
    Finally, we have inclusions 
    $$
    \Br(X) / \Br_1(X) \hookrightarrow \Br(X_K)/\Br_1(X_K)\hookrightarrow \Br(\overline{X})^{G_K},
    $$
    where $X_K$ is the base change of $X$ to $K$, so 
    $$
    \# (\Br(X) / \Br_1(X)) \leq \Nm(I).
    $$
\end{proof}

So our goal is to bound the norms of $k$-permissible ideals of $\co_E$, in terms of $E$ and $k$. 

\section{Making the totient bound explicit}
\label{sec-totient-bound}

Throughout this section, $k$ is a number field, $E$ is a CM field with maximal totally real subfield $F$, and $I$ is a $k$-permissible ideal of $\co_E$. Write $J = I \cap \co_F$. The bound on $\frac{\phi_E(I)}{\phi_F(J)}$ in Definition~\ref{defi-permissible-ideal} is quite inexplicit. In the current section, we obtain a weaker, but explicit bound. In Section~\ref{sec-deduce-norm-bound}, we will use this explicit bound to deduce bounds on $\Nm(I)$, hence on the transcendental Brauer group. 

\begin{lemma}[Bounding the ratio of subgroup indices]
    \label{lem-bound-on-subgroup-index-ratio}
    There exists a positive integer $m$ such that 
    $$
    \frac{[\co_E^\times : \co_F^\times]}{2^{[F:\Q] - 1}} = m \cdot \frac{[\co_E^\times : \co_E^I]}{[\co_F^\times : N_{E/F}(\co_E^I)]}. 
    $$
\end{lemma}
\begin{proof}
    By Dirichlet's unit theorem, there is a free abelian group $A$ of rank $[F : \Q] - 1$ such that $\co_F^\times = \{\pm 1\} \times A$. Also by Dirichlet's unit theorem, the group $A$ has finite index as a subgroup of $\co_E^\times$. Define 
    $$
    A^I = A \cap E^{I,1}.
    $$
    Since $E^{I,1}$ and $A$ are both subgroups of $E^\times$, their intersection $A^I$ is too. Moreover, we have a well-defined, injective map 
    $$
    A/A^I \to (\co_E/I)^\times,\quad [x] \mapsto [x],
    $$
    so $[A : A^I] \leq \phi_E(I)$, and therefore $A^I$ is free of rank $[F:\Q] - 1$. Since $N_{E/F}$ restricts to the squaring map on $F^\times$, we have 
    $$
    [A^I : N_{E/F}A^I] = 2^{[F:\Q] - 1},
    $$
    and $N_{E/F}A^I$ is a free abelian group of rank $[F:\Q] - 1$. Thus, all indices in the diagram 
    $$
        \begin{tikzcd}
            &                                                    & \co_E^\times                       &                                  \\
            & \co_E^I \arrow[ru, no head]                        &                                    & \co_F^\times \arrow[lu, no head] \\
    A^I \arrow[ru, no head] &                                                    & N_{E/F}\co_E^I \arrow[ru, no head] &                                  \\
            & N_{E/F}A^I \arrow[ru, no head] \arrow[lu, no head] &                                    &                                 
    \end{tikzcd}
    $$
    are finite, and we have 
    $$
    \frac{[\co_E^\times : \co_E^I]}{[\co_F^\times : N_{E/F}
    \co_E^I]} = \frac{[\co_E^\times : \co_F^\times]}{[A^I : N_{E/F}A^I]} \cdot \frac{[N_{E/F}\co_E^I : N_{E/F}A^I]}{[\co_E^I : A^I]}.
    $$
    The map $N_{E/F}$ descends to a well-defined group epimorphism 
    $$
    \co_E^I / A^I \to N_{E/F}(\co_E^I) / N_{E/F}(A^I),
    $$
    so 
    $$
    [N_{E/F}\co_E^I : N_{E/F}A^I] \mid [\co_E^I : A^I], 
    $$
    and the result follows since $[A^I : N_{E/F}A^I] = 2^{[F:\Q] - 1}$. 
\end{proof}
\begin{lemma}[Bounding ramification away from $J$]
    \label{lem-bound-on-e-E-F-J}
    We have 
    $$
    e(E/F, J) \leq 2^{[F : \Q] + \omega(d_{E/F})}. 
    $$
\end{lemma}
\begin{proof}
    Since $F$ is totally real, it has $[F:\Q]$ archimedean places, all of which ramify since $E$ is totally imaginary. It also has $\omega(d_{E/F})$ nonarchimedean places that ramify in $E$. The product $e(E/F,J)$ is at most the product of $e(v) = 2$ at all of these places, which is equal to $2^{[F : \Q] + \omega(d_{E/F})}$. 
\end{proof}
Recall from Section~\ref{subsec-nf-notation} that we write $I_2$ and $J_2$ for the 2-parts of $I$ and $J$, respectively. 
Valloni gives us the following bound on the size of $H^1(E^{I,1})$:
\begin{lemma}
    [Valloni]
    \label{lem-valloni-bound-on-tate-cohomology}
    We have 
    $$
    \#\big(H^1(E^{I,1})\big) \mid \hspace{1em} [(\co_E/I_2)^{\times, G} : (\co_F/J_2)^\times] \cdot \prod_{\p \mid J_2}e(\p),
    $$
    where the product is over primes $\p$ of $\co_F$ with $\p \mid J_2$, and $e(\p)$ is as defined in Section~\ref{subsec-nf-notation}. 
\end{lemma}
\begin{proof}
    This follows immediately from \cite[Proposition~10.5]{valloni-algorithm}. 
\end{proof}

Let $a_P$ be the unique integers such that 
$$
I = \prod_P P^{a_P},
$$
where the product is over all prime ideals of $\co_E$. Let $\p$ be a prime ideal of $\co_F$, and let $P$ be a prime of $\co_E$ lying over $\p$. Define 
$$
I_\p = \begin{cases}
    (P\overline{P})^{a_P} \quad\text{if $\p$ splits in $E$},
    \\
    P^{a_P}\quad\text{otherwise},
\end{cases}
$$
and 
$$
J_\p = \p^{\lceil\frac{a_P}{e(\p)}\rceil},
$$
where as usual $e(\p)$ is the ramification index of $\p$ in $E$. Note that $I_\p$ is well-defined since $I$ is conjugation-invariant. 
\begin{lemma}
    We have 
    $$
    I_2 = \prod_{\p\mid 2} I_\p ,\quad J_2 = \prod_{\p \mid 2} J_\p,
    $$
    where both products range over all primes of $\co_F$ containing $2$. 
\end{lemma}
\begin{proof}
    We have $I_2 = \prod_{\p\mid 2} I_\p$ by definition of the ideals $I_\p$. Since the ideals $I_\p$ are coprime, we have 
    $$
    I_2 = \bigcap_{\p\mid 2} I_\p.
    $$
    Since $J_2 = I_2 \cap \co_F$, it follows that 
    $$
    J_2 = \bigcap_{\p\mid 2} (I_\p \cap \co_F).
    $$
    Let $\p$ be a prime of $\co_F$ and let $P$ be a prime of $\co_E$ lying over $\p$. By considering the extension $E_P / F_\p$ of local fields, it is easy to see that 
    \begin{equation}
        \label{eqn-P-aP-cap-co_F}
    P^{a_P} \cap \co_F = \p^{\lceil \frac{a_P}{e(\p)}\rceil} = J_\p.
    \end{equation}
    If $\p$ does not split in $E$, then Equation~(\ref{eqn-P-aP-cap-co_F}) is precisely 
    $$
    J_\p = I_\p \cap \co_F. 
    $$
    If $\p$ does split in $E$, then $I_\p = P^{a_P} \cap \overline{P}^{a_P}$, so Equation~(\ref{eqn-P-aP-cap-co_F}) still implies that $J_\p = I_\p \cap \co_F$, and the result follows. 
\end{proof}
\begin{lemma}
    \label{lem-fixed-subgroup-index-product-formula}
    We have
    $$
    [(\co_E/I_2)^{\times, G} : (\co_F/J_2)^\times] = \prod_{\p \mid 2} [(\co_E/I_\p)^{\times, G} : (\co_F/J_\p)^\times]. 
    $$
\end{lemma}
\begin{proof}
    By the Chinese remainder theorem, we have a commutative diagram 
    $$
    \begin{tikzcd}
        \Big(\frac{\co_E}{I_2}\Big)^{\times, G} \arrow[r, "\cong"]                  & \bigoplus_{\p\mid 2} \Big(\frac{\co_E}{I_\p}\Big)^{\times, G}                 \\
        \Big(\frac{\co_F}{J_2}\Big)^{\times} \arrow[u, hook] \arrow[r, "\cong"'] & \bigoplus_{\p\mid 2} \Big(\frac{\co_F}{J_\p}\Big)^{\times}, \arrow[u, hook]
        \end{tikzcd}
    $$
    where every map is the most natural thing it could be, and the natural actions of $G$ are well-defined since $I_2$ and the $I_\p$ are conjugation-invariant. For each $\p$, the vertical inclusion on the right-hand side restricts to an inclusion
    $$
    \Big(\frac{\co_F}{J_\p}\Big)^{\times} \hookrightarrow \Big(\frac{\co_E}{I_\p}\Big)^{\times, G},
    $$
    and the result follows.
\end{proof}
\begin{lemma}
    \label{lem-fixed-subgroup-index-split}
    If $\p$ splits in $E$, then 
    $$
    [(\co_E/I_\p)^{\times, G} : (\co_F/J_\p)^\times]  = 1. 
    $$
\end{lemma} 
\begin{proof}
    By the Chinese remainder theorem, we have  
    $$
    (\co_E / I_\p)^\times \cong \Big(\frac{\co_E}{P^{a_P}}\Big)^\times \times \Big(\frac{\co_E}{\overline{P}^{a_P}}\Big)^\times. 
    $$
    The action of complex conjugation on $(\co_E/I_\p)^\times$ corresponds to an action on the right-hand side by 
    $$
    ([x],[y])\mapsto ([\overline{y}], [\overline{x}]). 
    $$
    It is easy to see that the fixed subgroup
    $$
    \Big(\Big(\frac{\co_E}{P^{a_P}}\Big)^\times \times \Big(\frac{\co_E}{\overline{P}^{a_P}}\Big)^\times\Big)^G 
    $$
    is equal to the image of the natural inclusion 
    $$
    \iota:\Big(\frac{\co_F}{\p^{a_P}}\Big)^\times \to \Big(\frac{\co_E}{P^{a_P}}\Big)^\times \times \Big(\frac{\co_E}{\overline{P}^{a_P}}\Big)^\times,
    $$
    so we are done. 
\end{proof}
\begin{lemma}
    \label{lem-fixed-subgroup-index-inert}
    If $\p$ is inert in $E$, then 
    $$
    [(\co_E/I_\p)^{\times, G} : (\co_F/J_\p)^\times]  = 1. 
    $$
\end{lemma}
\begin{proof}
    It is easy to see that we may instead prove the analogous result for quadratic, unramified extensions $E/F$ of $p$-adic fields. That is, we let $F$ be a finite-degree extension of $\Q_p$ and let $E/F$ be an unramified quadratic extension. Let $\p$ and $P$ be the maximal ideals of $\co_F$ and $\co_E$, respectively. Then, for each positive integer $a$, we need to show that 
    $$
    [(\co_E/P^a)^{\times, G} : (\co_F/\p^a)^\times] = 1,
    $$
    where $G = \Gal(E/F)$. The proof of \cite[Part~1, Chapter~3, Proposition~3]{lang1994algebraic} tells us that $\co_E = \co_F[\theta]$ for some $\theta \in \co_E$ such that $[\theta]$ is a generator of the residue field extension $\F_E/\F_F$. Let 
    $$
    \alpha \in (\co_E/P^a)^{\times, G}.
    $$
    Then $\alpha = [m + n\theta]$ for elements $m,n\in \co_F$. Write $\overline{\theta}$ for the conjugate of $\theta$ by $G$. Since $\alpha$ is fixed by $G$, we have 
    $$
    m + n\theta \equiv m + n\overline{\theta} \pmod{P^a},
    $$
    so 
    $$
    n(\theta - \overline{\theta}) \in P^a. 
    $$
    Since $[\theta] \not \in \F_F$ and the restriction map 
    $$
    \Gal(E/F) \to \Gal(\F_E/\F_F) 
    $$
    is surjective, we have $v_E(\theta - \overline{\theta}) = 0$, so $n \in \p^a$, and therefore $\alpha = [m]$, so $\alpha$ is in the image of the inclusion 
    $$
    (\co_F/\p^a)^\times \hookrightarrow (\co_E/P^a)^\times,
    $$
    and the result follows. 
\end{proof}
\begin{lemma}
    \label{lem-fixed-subgroup-index-ramified}
    If $\p$ lies over $2$ and ramifies in $E$, then 
    $$
    [(\co_E/I_\p)^{\times, G} : (\co_F/J_\p)^\times] = \begin{cases}
        \Nm(\p)^{\lfloor\frac{a_P}{2}\rfloor}\quad\text{if $a_P < v_\p(d_{E/F})$},
        \\
        \Nm(\p)^{\frac{v_\p(d_{E/F})}{2} - 1} \quad\text{if $v_\p(d_{E/F})$ is even and $a_P \geq v_\p(d_{E/F})$ is odd,}
        \\
        \Nm(\p)^{\frac{v_\p(d_{E/F})}{2}} \quad\text{if $v_\p(d_{E/F})$ is even and $a_P\geq v_\p(d_{E/F})$ is even},
        \\
        \Nm(\p)^{e(\p\mid 2)}\quad\text{if $v_\p(d_{E/F}) = 2e(\p\mid 2) + 1$ and $a_P\geq v_\p(d_{E/F})$} .
    \end{cases}
    $$
\end{lemma}
\begin{proof}
    It is easy to see that we may instead prove the analogous result in the case where $E/F$ is a totally ramified quadratic extension of $2$-adic fields. That is, let $F$ be a finite-degree extension of the $2$-adic numbers $\Q_2$ and let $E/F$ be a totally ramified quadratic field extension. Let $\p$ and $P$ be the maximal ideals of $\co_F$ and $\co_E$, respectively. Then, for each positive integer $a$, we need to compute the index 
    $$
    [(\co_E/P^a)^{\times, G} : (\co_F/\p^{\lceil\frac{a}{2}\rceil})^\times].
    $$
    Write $e_F$ for the absolute ramification index of $F$, so $e_F$ corresponds to $e(\p\mid 2)$ in the statement of the lemma.
    We claim that there is an element $\rho \in \co_E$ such that the following three statements are true:
    \begin{enumerate}
        \item $v_E(\rho) = 1$.
        \item $\co_E = \co_F \oplus \co_F\cdot \rho$. 
        \item $v_E(\rho - \overline{\rho}) = v_F(d_{E/F})$. 
    \end{enumerate}
     By \cite[Lemma~4.3]{tunnell}, either $v_F(d_{E/F})$ is even with $2 \leq v_F(d_{E/F}) \leq 2e_F$, or $v_F(d_{E/F}) = 2e_F + 1$. In the former case, the existence of $\rho$ follows from \cite[Lemmas~3.4-3.5]{cdo}. If $v_F(d_{E/F}) = 2e_F + 1$, then taking $p=2$ in \cite[Theorem~2.4]{hecke-theorem}, we see that $E = F(\sqrt{\pi_F})$ for a uniformiser $\pi_F$ of $F$. Let $\rho = \sqrt{\pi_F}$. Since the polynomial $X^2 - \pi_F$ is Eisenstein over $F$, we have $\co_E = \co_F\oplus \co_F\cdot \rho$. The other two properties obviously hold. 
    
    Given the existence of $\rho$, we obtain a well-defined bijection 
    $$
    \varphi : (\co_F/\p^{\lceil \frac{a}{2}\rceil})^\times \times \Big(\p^{\max\big\{0, \big\lceil \frac{a - v_F(d_{E/F})}{2}\big\rceil\big\}} / \p^{\lfloor \frac{a}{2}\rfloor}\Big) \to \Big(\frac{\co_E}{P^{a}}\Big)^{\times, G},\quad ([m],[n])\mapsto [m + n\rho],
    $$
    and the result follows. 
\end{proof}
\begin{lemma}
    \label{lem-fixed-subgroup-index-divides-norm}
    We have 
    $$
    [(\co_E/I_2)^{\times, G} : (\co_F/J_2)^\times]^2 \mid \Nm(d_{E/F, 2}). 
    $$
\end{lemma}
\begin{proof}
    This is immediate from Lemmas~\ref{lem-fixed-subgroup-index-product-formula}, \ref{lem-fixed-subgroup-index-split}, \ref{lem-fixed-subgroup-index-inert}, and \ref{lem-fixed-subgroup-index-ramified}.  
\end{proof}

\begin{lemma}[Bounding Tate cohomology]
    \label{lem-bound-on-tate-cohomology}
    We have 
    $$
    \# \big(H^1(E^{I,1})\big)^2 \mid 2^{2\omega(d_{E/F, 2})} \cdot \Nm(d_{E/F,2}).
    $$
\end{lemma}
\begin{proof}
    Since each ramified prime dividing $J_2$ also divides $d_{E/F,2}$, we have 
    $$
    \prod_{\p \mid J_2} e(\p) \mid 2^{\omega(d_{E/F,2})}. 
    $$
    The result then follows from Lemmas~\ref{lem-valloni-bound-on-tate-cohomology} and \ref{lem-fixed-subgroup-index-divides-norm}. 
\end{proof}

\begin{corollary}
    \label{cor-explicit-totient-bound}
    For each $k$-permissible ideal $I$ of $\co_E$, we have 
    $$
    \frac{\phi_E(I)}{\phi_F(J)} \leq [k : \Q] \cdot M_E,
    $$
    where $M_E$ is the explicit constant defined in Section~\ref{sec-intro}.
\end{corollary}
\begin{proof}
    This is immediate from Lemmas~\ref{lem-bound-on-subgroup-index-ratio}, \ref{lem-bound-on-e-E-F-J}, and \ref{lem-bound-on-tate-cohomology}. 
\end{proof}

\section{Deducing a bound on the norm}
\label{sec-deduce-norm-bound}

Again, throughout this section, we fix a number field $k$, a CM field $E$, and a $k$-permissible ideal $I$ of $\co_E$. Set $J = I \cap F$ and, for each prime $\p$ of $F$, define the ideals $I_\p$ and $J_\p$ as in Section~\ref{sec-totient-bound}.

\begin{lemma}
    \label{lem-totient-function-formula}
    Let $M$ be a number field and let $\mathfrak{a}$ be an ideal of $\co_M$. Let $\mathfrak{a}$ have prime factorisation 
    $$
    \mathfrak{a} = \p_1^{a_1}\ldots \p_r^{a_r},
    $$
    where the $\p_i$ are distinct prime ideals of $\co_M$ and the $a_i$ are positive integers. Then 
    $$
    \phi_M(\mathfrak{a}) = \prod_{i=1}^r \Nm(\p_i)^{a_i-1}(\Nm(\p_i) - 1). 
    $$
\end{lemma}
\begin{proof}
    This follows easily from the Chinese remainder theorem. 
\end{proof}
For prime ideals $P$ of $\co_E$, define the integers $a_P$ as in Section~\ref{sec-totient-bound}.
\begin{lemma}
    \label{lem-totient-ratio-at-p}
    Let $\p$ be a prime ideal of $\co_F$ dividing $J$, and let $P$ be a prime ideal of $\co_E$ lying over $\p$. Then we have
    $$
    \frac{\phi_E(I_\p)}{\phi_F(J_\p)} = \begin{cases}
        \Nm(P)^{a_P - 1}(\Nm(P) - 1)  \quad \text{if $\p$ splits in $E$},
        \\
        \Nm(P)^{\frac{a_P-1}{2}}(\sqrt{\Nm(P)} + 1)\quad \text{if $\p$ is inert in $E$},
        \\
        \Nm(P)^{\lfloor \frac{a_P}{2}\rfloor} \quad \text{if $\p$ is ramified in $E$}.
    \end{cases}
    $$
\end{lemma}
\begin{proof}
    This is immediate from Lemma~\ref{lem-totient-function-formula} and the definitions of $I_\p$ and $J_\p$. 
\end{proof}

\begin{lemma}
    \label{lem-bounded-prod-of-Np-1-implies-bounded-prod-of-quotients}
    Let $S$ be a finite set of prime ideals of $\co_E$, and suppose that 
    $$
    \prod_{P \in S} (\Nm(P) - 1) \leq T,
    $$
    for some real number $T$. Then we have 
    $$
    \prod_{P \in S} \frac{\Nm(P)}{\Nm(P) - 1} \leq \Psi\Big(\Phi^{-1}\Big(\frac{\log T}{[E : \Q]}\Big)\Big)^{[E:\Q]}.
    $$
\end{lemma}
\begin{proof}
    Let 
    $
    S = \{P_1,\ldots, P_r\},
    $
    where $\Nm(P_i) \leq \Nm(P_{i+1})$ for each $i$. Since $E$ has at most $[E:\Q]$ primes lying over each rational prime, we have 
    $$
    \Nm(P_i) \geq p_{\lfloor\frac{i-1}{[E:\Q]}\rfloor}
    $$
    for each $i$, and therefore 
    $$
    \prod_{P \in S} (\Nm(P) - 1) \geq \Big(\prod_{j=0}^{\lfloor \frac{\# S}{[E : \Q]} \rfloor - 1} (p_j - 1)\Big)^{[E : \Q]}.
    $$
    It follows that 
    $$
    [E : \Q]\cdot  \Phi\Big(\Big\lfloor\frac{\# S}{[E:\Q]}\Big\rfloor - 1\Big) \leq \log T,
    $$
    so 
    $$
    \Big\lfloor\frac{\# S}{[E:\Q]} \Big\rfloor - 1\leq \Phi^{-1}\Big(\frac{\log T}{[E:\Q]}\Big).
    $$
    Since $\Nm(P_i) \geq p_{\lfloor \frac{i - 1}{[E:\Q]}\rfloor}$ for each $i$, it follows that 
    $$
    \frac{\Nm(P_i)}{\Nm(P_i) - 1} \leq \frac{p_{\lfloor \frac{i-1}{[E:\Q]}\rfloor}}{p_{\lfloor \frac{i-1}{[E:\Q]}\rfloor} - 1},
    $$
    so 
    $$
    \prod_{P \in S} \frac{\Nm(P)}{\Nm(P) - 1} \leq \Big(\prod_{j=0}^{\lceil \frac{\# S}{[E : \Q]} \rceil - 1} \frac{p_j}{p_j - 1}\Big)^{[E : \Q]} = \Psi\Big(\Big\lceil \frac{\# S}{[E:\Q]}\Big\rceil - 2\Big)^{[E:\Q]}.
    $$
    The result follows since 
    $$
    \Big\lceil \frac{\# S}{[E:\Q]} \Big\rceil - 2 \leq \Big\lfloor \frac{\# S}{[E:\Q]}\Big\rfloor - 1 \leq \Phi^{-1}\Big(\frac{\log T}{[E:\Q]}\Big). 
    $$
\end{proof}

\begin{proof}
    [Proof of Theorem~\ref{thm-bound-with-Phi-and-Psi}]
    By Corollary~\ref{cor-brauer-bounded-by-Nm-I}, we have 
    $$
    \#\big(\Br(X) / \Br_1(X)\big) \leq \Nm(I)  
    $$
    for some $k$-permissible ideal $I$ of $\co_E$. Using multiplicativity of the totient, Corollary~\ref{cor-explicit-totient-bound}, and Lemma~\ref{lem-totient-ratio-at-p}, it is easy to see that 
    \begin{align*}
    \Nm(I) &\leq \Big(\frac{\phi_E(I)}{\phi_F(J)}\Big)^2 \cdot \prod_{\substack{P \mid I \\ \mathrm{split}}} \frac{\Nm(P)}{\Nm(P) - 1} \cdot \Nm(\rad(d_{E/F}))
    \\
    &\leq [k:\Q]^2\cdot M_E^2 \cdot \prod_{\substack{P \mid I \\ \mathrm{split}}} \frac{\Nm(P)}{\Nm(P) - 1} \cdot \Nm(\rad(d_{E/F})),
    \end{align*}
    where the product is over all prime ideals $P$ dividing $I$ that are split in $E/F$ (i.e. over all $P\mid I$ with $\overline{P} \neq P$). Also using Corollary~\ref{cor-explicit-totient-bound} and Lemma~\ref{lem-totient-ratio-at-p}, we have 
    $$
    \prod_{\substack{P \mid I \\ \mathrm{split}}} (\Nm(P) - 1) \leq \Big(\frac{\phi_E(I)}{\phi_F(J)}\Big)^2 \leq [k : \Q]^2 \cdot M_E^2,
    $$
    so Lemma~\ref{lem-bounded-prod-of-Np-1-implies-bounded-prod-of-quotients} implies that 
    $$
    \prod_{\substack{P \mid I \\ \mathrm{split}}} \frac{\Nm(P)}{\Nm(P) - 1} \leq \Psi\Big(\Phi^{-1}\Big(\frac{2\log([k:\Q]M_E)}{[E : \Q]}\Big)\Big)^{[E:\Q]},
    $$
    and the result follows from Corollary~\ref{cor-brauer-bounded-by-Nm-I}. 
\end{proof}

\begin{proof}
    [Proof of Theorem~\ref{thm-bounding-Psi-Phi-inverse}]
    Let $n = \Phi^{-1}(t)$, so $n$ is a nonnegative integer. We have $\log(p_0 - 1) = 0$ and $\log(p_i - 1) \geq \log 2$ for all $i \geq 1$, so  
    $$
    t \geq \Phi(n) \geq n\log 2,
    $$
    so $n \leq \frac{t}{\log 2}$. If $n \leq L_\delta$, then obviously
    $$
    \Psi(n) \leq \Psi(L_\delta) \leq \Psi(L_\delta)\cdot (1+\delta)^n. 
    $$
    If $n > L_\delta$, then 
    \begin{align*}
        \Psi(n) &= \Psi(L_\delta) \cdot \prod_{i=L_\delta + 2}^{n+1} \frac{p_i}{p_i-1} 
        \\
        &\leq \Psi(L_\delta)\cdot (1 + \delta)^{n - L_\delta}
        \\
        &\leq \Psi(L_\delta) \cdot (1 + \delta)^{n+1},
    \end{align*}
    so we are done. 
\end{proof}
\begin{lemma}
    \label{lem-eratosthenes-time-complexity}
    Let $T$ be a positive real number. Using the sieve of Eratosthenes, one can compute all prime numbers less than $T$ with time complexity $O(T\log\log T)$. 
\end{lemma}
\begin{proof}
    See \cite[Page~4]{sorensen-sieve}.
\end{proof}

\begin{proof}
    [Proof of Theorem~\ref{thm-time-complexity-C-eps-delta}]
    We must find the smallest positive integer $L_\delta$ such that 
    $$
    p_{L_\delta + 2} \geq 1 + \frac{1}{\delta}.
    $$
    Using the well-known bounds 
    $$
    n(\log n + \log \log n - 1) < p_{n-1} < n(\log n + \log\log n),
    $$
    valid for $n\geq 6$,
    we can quickly find an $L_0$ such that $L_0 \log L_0 \sim \frac{1}{\delta}$ and $p_{L_0 + 2} \geq 1 + \frac{1}{\delta}$. Since $L_0\log L_0 \sim \frac{1}{\delta}$, Lemma~\ref{lem-eratosthenes-time-complexity} tells us that we can compute $p_i$ for every $i \leq L_0$ with time complexity $O(L_0\log\log L_0)$. Having computed $p_i$ for each $i \leq L_0$, we can find the smallest $i$ with $p_i \geq 1 + \frac{1}{\delta}$ by performing a binary search on $\Z \cap [0, L_0]$, which has time complexity $O(\log L_0)$. Clearly 
    $$
    \log L_0 \ll L_0 \log\log L_0 \ll L_0 \log L_0 \sim \frac{1}{\delta},
    $$
    so we can determine $L_\delta$ with the claimed time complexity. Given that the required $p_i$ have already been computed, we can then compute $\Psi(L_\delta)$ with time complexity $O(L_\delta)$, and we know that $L_\delta \ll \frac{1}{\delta}$, so we are done. 
\end{proof}
\begin{proof}
    [Proof of Corollary~\ref{cor-to-bound-with-Psi-Phi-inverse}]
    This is immediate from Theorems~\ref{thm-bound-with-Phi-and-Psi} and \ref{thm-bounding-Psi-Phi-inverse}. 
\end{proof}
\begin{proof}
    [Proof of Theorem~\ref{thm-explicit-bound-without-phi-and-psi}]
    This follows from Corollary~\ref{cor-to-bound-with-Psi-Phi-inverse}, where we take $\delta = 1/2$. 
\end{proof}

\printbibliography

\end{document}